\documentclass[12pt]{article}
\usepackage[T1]{fontenc}
\usepackage[cp1250]{inputenc}
\usepackage{lmodern}

\usepackage[english]{babel}
\usepackage{latexsym}
\usepackage{amsmath}
\usepackage{indentfirst}
\usepackage{amsthm}
\usepackage{amssymb}
\usepackage{amsfonts}
\usepackage{stackrel}
\usepackage{dsfont}
\usepackage{tikz}
\usepackage{slashbox}
\usepackage[mathscr]{eucal}
\usepackage{authblk}

\newtheorem{thm}{Theorem}
\newtheorem{lemma}[thm]{Lemma}
\newtheorem{cor}[thm]{Corollary}

\newcommand{\reals}{\mathbb{R}}
\newcommand{\naturals}{\mathbb{N}}
\newcommand{\integers}{\mathbb{Z}}
\newcommand{\complex}{\mathbb{C}}

\newcommand{\A}{\mathcal{A}}
\newcommand{\ev}{\text{ev}}
\newcommand{\Inv}{\text{Inv}}

\begin{document}
\title{Weighted $L^p-$spaces on nilpotent, locally compact groups}
\author{Mateusz Krukowski}
\affil{\L\'od\'z University of Technology, Institute of Mathematics, \\ W\'ol\-cza\'n\-ska 215, \
90-924 \ \L\'od\'z, \ Poland}
\maketitle

\begin{abstract}
Our paper begins with a revision of spectral theory for commutative Banach algebras, which enables us to prove the $L^p_{\omega}-$conjecture for locally compact abelian groups. We follow an alternative approach to the one known in the literature. In particular, we do not resort to any structural theorems for locally compact groups at this stage. Subsequently, we discuss nilpotent, locally compact groups. The climax of the paper is the proof of the $L^p_{\omega}-$conjecture for these groups.
\end{abstract}

\smallskip
\noindent 
\textbf{Keywords : } abstract harmonic analysis, locally compact groups, nilpotent groups, weights\\
\vspace{0.2cm}
\\
\textbf{AMS Mathematics Subject Classification : } 43A15

\section{Introduction}

Throughout the paper, $G$ always stands for a locally compact group $G$ with a left Haar \mbox{measure $\mu$} (we will always assume that the group is Hausdorff). In Section \ref{section:Spectraltheory}, we will additionally require the groups to be abelian. Section \ref{section:Nilpotentgroups} will focus on the nilpotent, locally compact groups. 

The convolution of two measurable functions \mbox{$f,g : G\rightarrow \complex$} is given by
$$f\star g(x) = \int_G\ f(y)\ g\left(y^{-1}x\right)\ d\mu(y).$$

\noindent
The classical $L^p-$conjecture states that if $L^p(G)$ is closed under convolution for $p>1$, then the group $G$ must be compact. The first results relating to this conjecture date back to the papers of Urbanik \cite{Urbanik} and \.Zelazko \cite{ZelazkoonLpalgebras}. However, it was Rajagopalan who formulated the conjecture in his PhD thesis in 1963. In the next 25 years, the conjecture has been studied in a variety of locally compact groups, comp. \cite{GaudetGamlen}, \cite{Johnson}, \cite{Rajagopalandiscrete}, \cite{RajagopalanI}, \cite{RajagopalanII}, \cite{RajagopalanZelazko}, \cite{RickertL2}, \cite{RickertLp}. At last, the conjecture was resolved affirmatively in 1990 by Sadashiro Saeki, comp. \cite{Saeki}. 

The revival of the subject came in 2010, when Abtahi, Nasr-Isfahani and Rejali considered the following variant of the conjecture (comp. \cite{ANR}): Let $p>1$ and let $L^p_{\omega}(G)$ be the Banach space of measurable functions $f:G\rightarrow \complex$ such that 
$$\int_G\ (f\cdot\omega)^p\ d\mu < \infty,$$

\noindent
where $\omega:G\rightarrow \reals_+$ is a measurable and submultiplicative function. The norm in this space is $\|f\|_{p,\omega} = \|f\cdot\omega\|_p$. The authors conjectured that if $L^p_{\omega}(G)$ is closed under convolution, then $G$ is $\sigma-$compact. They proved this to be true for locally compact abelian groups (comp. \cite{ANR}) and for general locally compact groups, if $p>2$ (comp. \cite{ANRp2}). 

In this paper, we provide a partial answer to the following question:
\begin{center}
\textit{under what conditions on the group $G$ and the weight $\omega$, does $L^p_{\omega}(G)\star L^p_{\omega}(G)\subset L^p_{\omega}(G)$ imply the compactness of $G$?}
\end{center} 

\noindent
In order to avoid any misapprehension: a function $\omega : G\rightarrow \reals_+$ is called a \textit{weight} (comp. \cite{StegmanReiter}, p. 119) if it is measurable and \textit{submultiplicative}:
$$\forall_{x,y\in G}\ \omega(xy)\leq \omega(x)\ \omega(y).$$

\noindent
A weight $\omega$ is called \textit{diagonally bounded} if there exists $M>0$ such that 
$$\forall_{x\in G}\ \omega(x)\ \omega\left(x^{-1}\right) < M.$$

\noindent
Although the most natural examples of weights like \mbox{$\omega(x) = \exp(|x|)$} or \mbox{$\omega(x) = \left(1+|x|\right)^{\alpha}$}, $\alpha \geq 0$ are not diagonally bounded, this is still a sufficiently large family to be interesting to study. The most obvious example is $\omega(x) = \exp(x)$ on $\reals$, which is diagonally bounded since 
$$\omega(x)\ \omega(-x) = \exp(x)\ \exp(-x) = 1.$$

\noindent
In a similar vein, if $G$ is an arbitrary locally compact group, then $\omega:G\times \integers \rightarrow \reals_+$ given by $\omega(x,n) = 2^n$ is diagonally bounded. If we want to stray from the exponential-like behaviour, we may consider $\omega:G\times\reals_+\rightarrow\reals_+$ (comp. \cite{Maligranda}) given by
$$\omega(x,y) = y^{\alpha} \hspace{0.4cm}\text{or}\hspace{0.4cm} \omega(x,y) = y^{\alpha}\left(1+\sin{\ln{y}}\right),\hspace{0.4cm} \alpha\geq0.$$

As far as the organization of the paper is concerned, Section \ref{section:Spectraltheory} begins with a recap on the spectral theory of commutative Banach algebras. We deliberately avoid assuming the algebra to be unital since $L^p_{\omega}(G)$ may not be unital. We recall the notions of spectrum of an element and the spectral radius, and how they relate to the family of linear and multiplicative functionlas, the so-called spectrum of the algebra. We closely follow the exposition in Kaniuth's 'A Course on Commutative Banach Algebras' (comp. \cite{Kaniuth}). 

With these tools at hand, we can prove Theorem \ref{LpconjectureforLCA}. This is a generalization of the main theorem in \cite{Urbanik} and a variation of Theorem 4.5 in \cite{ANR}. Our proof does not use any structural theorems for locally compact groups. 

In Section \ref{section:Nilpotentgroups}, we shift our attention to nilpotent, locally compact groups. We discuss the upper central series and its properties. The climax of the paper is Theorem \ref{nilpotentgroups}, which establishes $L^p_{\omega}-$conjecture in the mentioned setting.

\section{Spectral theory and the $L^p_{\omega}(G)-$conjecture for locally compact abelian groups}
\label{section:Spectraltheory}

The crucial part of Theorem \ref{LpconjectureforLCA} (the main result in this section) is the fact that the spectrum of the commutative Banach algebra $L^p_{\omega}(G)$, denoted by $\Delta\left(L^p_{\omega}(G)\right)$, contains a nonzero element. In order to justify this claim, we briefly recall the notions of spectral theory. Let $\A$ be an arbitrary commutative Banach algebra. Let us emphasize the fact that we do not assume $\A$ to be unital, as $L^p_{\omega}(G)$ may not have a unit. 

In spite of what we said above, the property of having a unit usually simplifies the theory. Hence, the notion of unitization $\A_e = \A\times \complex$ of $\A$ is crucial. It is a simple exercise that with the following operations (comp. \cite{Kaniuth}, p. 6):
\begin{gather*}
\forall_{\substack{x,y\in \A\\ \alpha,\beta\in\complex}}\ (x,\alpha) + (y,\beta) = (x+y,\alpha+\beta),\\
\beta(x,\alpha) = (\beta x,\beta\alpha),\\
(x,\alpha)(y,\beta) = (xy+\alpha y + \beta x, \alpha \beta),\\
\|(x,\alpha)\| = \|x\| + |\alpha|,
\end{gather*}

\noindent
$\A_e$ becomes a unital, commutative Banach algebra. The element $(0,1)$ is an identity in $\A_e$ and every element $x\in \A$ can be identified with $(x,0)\in \A_e$. Moreover, by $\Inv(\A_e)$ we denote the family of those elements $x\in \A_e$, which possess an inverse, i.e. there exists $y \in A_e$ such that $xy = yx = (0,1)$. 

The \textit{spectrum of the element} $x\in \A$ is the set
$$\sigma(x) =\bigg\{\lambda \in\complex\ :\ \lambda (0,1) - (x,0)\not\in \Inv(\A_e)\bigg\}.$$

\noindent
The \textit{spectral radius} of $x$ is given by
$$\rho(x) = \max\bigg\{|\lambda| \ : \ \lambda \in \sigma(x)\bigg\}.$$

\noindent
We have deliberately used '$\max$' instead of '$\sup$' to underline the fact that $\sigma(x)$ is a nonempty, compact subset of $\complex$ \mbox{(comp. Theorem 1.2.8 in \cite{Kaniuth}, p. 10)}. Furthermore, the same theorem establishes a convenient way to compute the spectral radius:
\begin{gather}
\rho(x) = \lim_{n\rightarrow\infty}\ \|x^n\|^{\frac{1}{n}}.
\label{computespectralradius}
\end{gather}

The next result reveals the intricate link between the spectrum of the element and the spectrum of the whole algebra:

\begin{thm}(Theorem 2.2.5 in \cite{Kaniuth}, p. 54)\\
Let $\A$ be a commutative Banach algebra (not necessarily unital). If $x \in \A$, then
$$\sigma(x)\backslash \{0\} \subset \ev_x\left(\Delta(\A)\right) \subset \sigma(x),$$

\noindent
where $\ev_x$ is the evaluation map given by $\ev_x(\chi) = \chi(x).$
\label{Kaniuththeorem}
\end{thm}

\begin{cor}
If there exists $x\in \A$ such that 
\begin{gather}
\lim_{n\rightarrow\infty}\ \|x^n\|^{\frac{1}{n}} > 0,
\label{limxn}
\end{gather}

\noindent
then $\Delta(\A) \neq \{0\}$.
\label{corollaryonspectrum} 
\end{cor}
\begin{proof}
Suppose, for the sake of contradiction, that $\Delta(\A) = \{0\}$. \mbox{Assumption (\ref{limxn})} implies, due to (\ref{computespectralradius}), that $\sigma(x)\backslash \{0\}\neq \emptyset$. Consequently, Theorem \ref{Kaniuththeorem} yields a contradiction as $\ev_x(\Delta(\A)) = \ev_x(\{0\})=0$. We conclude that $\Delta(\A) \neq \{0\}$.
\end{proof}

Before we prove the final theorem of this section, we need a technical lemma. Note, that the notation $A\lesssim B$ for $A,B\geq 0$ means that there exists a constant $C>0$ such that $A \leq C B$.

\begin{lemma}(comp. \cite{Krukowski})\\
Let $p,q,r \in [1,\infty)$. If $L^p_{\omega}(G)\star L^q_{\omega}(G)\subset L^r_{\omega}(G)$, then
$$\|f\star g\|_{r,\omega} \lesssim \|f\|_{p,\omega}\ \|g\|_{q,\omega}.$$

\noindent
In particular, if $\omega \equiv 1$, then
\begin{gather}
\|f\star g\|_r \lesssim \|f\|_p\ \|g\|_q.
\label{fgC0fg}
\end{gather}
\label{fstarglesssimfg}
\end{lemma}

We are ready to present the crowning result of the current section, which is a variation on Theorem 4.5 in \cite{ANR}. It is noteworthy that the proof does not use any structural theorems for locally compact groups. 

\begin{thm}
Let $G$ be a locally compact abelian group and let $L^p_{\omega}(G)$ be closed under convolution, where $\omega$ is diagonally bounded. Then $G$ is compact.
\label{LpconjectureforLCA}
\end{thm}
\begin{proof}
Let $K$ be a symmetric and compact neighbourhood of an identity $e\in G$ with nonzero measure. Observe that for $x\in K$ we have
\begin{equation}
\begin{split}
&\mathds{1}_{K^2}^{n+1}(x) = \int_G\ \mathds{1}_{K^2}^n(y)\ \mathds{1}_{K^2}\left(y^{-1}x\right)\ d\mu(y) \\
&\geq \int_K\ \mathds{1}_{K^2}^n(y)\ \mathds{1}_{K^2}\left(y^{-1}x\right)\ d\mu(y) = \int_K\ \mathds{1}_{K^2}^n(y)\ d\mu(y).
\end{split}
\end{equation}

\noindent
An easy induction shows that for $n\in\naturals$ we have
\begin{gather}
\forall_{x\in K}\ \mathds{1}_{K^2}^{n+1}(x) \geq \mu^n(K).
\label{mathds1inequality}
\end{gather}

\noindent
Consequently, we obtain
$$\|\mathds{1}_{K^2}^{n+1}\|_{p,\omega} = \left(\int_K\ \left(\mathds{1}_{K^2}^{n+1}\cdot\omega\right)^p\ d\mu\right)^{\frac{1}{p}} \stackrel{(\ref{mathds1inequality})}{\geq} \mu^n(K)\ \left(\int_K\ \omega^p\ d\mu\right)^{\frac{1}{p}},$$

\noindent
where $\left(\int_K\ \omega^p\ d\mu\right)^{\frac{1}{p}}$ has a finite value due to Lemma 2.1 in \cite{Kuznetsova}. This implies that 
$$\lim_{n\rightarrow\infty}\ \|\mathds{1}_{K^2}^n\|_{p,\omega}^{\frac{1}{n}} \geq \mu(K) > 0.$$

\noindent
By Corollary \ref{corollaryonspectrum}, we conclude that $\Delta\left(L^p_{\omega}(G)\right)$ contains a nonzero element, which we denote by $\chi$.

The dual of $L^p_{\omega}(G)$ is $L^{p'}_{\frac{1}{\omega}}(G)$, where $\frac{1}{p}+\frac{1}{p'} = 1$. Hence there exists a nonzero $h\in L^{p'}_{\frac{1}{\omega}}(G)$ such that 
$$\chi(f) = \int_G\ f\cdot h\ d\mu.$$

\noindent
In the sequel, $C_c^+(G)$ stands for the family of compactly supported, continuous and nonnegative functions on $G$. Observe that 
\begin{equation*}
\begin{split}
\forall_{f,g \in C_c^+(G)}\ &\bigg|\int_G\ f\star g(x)\ h(x)\ d\mu(x)\bigg| \leq \int_G\ |f\star g(x)\ h(x)|\ d\mu(x)\\ &\stackrel{\text{H\"{o}lder inequality}}{\leq} \|f\star g\|_{p,\omega}\ \|h\|_{p',\frac{1}{\omega}} \stackrel{\text{Lemma}\ \ref{fstarglesssimfg}}{\lesssim} \|f\|_{p,\omega}\ \|g\|_{p,\omega}\ \|h\|_{p',\frac{1}{\omega}} < \infty.
\end{split}
\end{equation*}

\noindent
This means that the integrals, which appear below, are finite.

For $f,g \in C_c^+(G)$ we have
\begin{gather*}
\int_G\ f\star g(x)\ h^+(x)\ d\mu(x) = \int_G\ \left(\int_G\ f(y)\ g\left(y^{-1}x\right)\ d\mu(y)\right)\ h^+(x)\ d\mu(x) \\
= \int_G\ \int_G\ f(y)\ g\left(y^{-1}x\right)\ h^+(x)\ d\mu(y)\ d\mu(x) \\
\stackrel{\text{Tonelli thm}}{=} \int_G\ \int_G\ f(y)\ g\left(y^{-1}x\right)\ h^+(x)\ d\mu(x)\ d\mu(y) \\
\stackrel{x\mapsto yx}{=} \int_G\ \int_G\ f(y)\ g(x)\ h^+(yx)\ d\mu(x)\ d\mu(y),
\end{gather*} 

\noindent
where $h^+ = \max{(h,0)}$. We conclude that 
\begin{gather}
\forall_{f,g\in C_c^+(G)}\ \chi(f\star g) =  \int_G\ \int_G\ f(y)\ g(x)\ h^+(yx)\ d\mu(x)\ d\mu(y).
\label{chifstarg1}
\end{gather}

On the other hand, we have 
\begin{gather}
\forall_{f,g\in C_c^+(G)}\ \chi(f)\ \chi(g) = \int_G\ \int_G\ f(y)\ g(x)\ h^+(y)\ h^+(x)\ d\mu(x)\ d\mu(y).
\label{chifstarg2}
\end{gather}

\noindent
Comparing (\ref{chifstarg1}) and (\ref{chifstarg2}), we establish that for $f,g\in C_c^+(G)$ we have
$$\int_G\ \int_G\ f(y)\ g(x)\ \bigg(h^+(yx) - h^+(y)\ h^+(x)\bigg)\ d\mu(x)\ d\mu(y) = 0.$$

\noindent
We apply the same reasoning as in the proof of the fundamental lemma of calculus of variations (comp. \cite{GelfandFomin}, p. 22) to conclude that $h^+(yx) = h^+(y)\ h^+(x)$ for a.e. $x,y\in G$. An analogous reasoning works for $h^-=\max{(-h,0)}$, so 
\begin{gather}
\forall_{a.e.\ x,y\in G}\ h(xy) = h(x)\ h(y).
\label{hmultiplicative}
\end{gather}

Finally, we have
\begin{gather*}
\int_G\ \left(\frac{|h|(x)}{\omega(x)}\right)^{p'}\ d\mu(x) \stackrel{x\mapsto xy}{=} \int_G\ \left(\frac{|h|(xy)}{\omega(xy)}\right)^{p'}\ d\mu(x) \\
\stackrel{(\ref{hmultiplicative})}{=} |h|(y)^{p'}\ \int_G\ \left(\frac{|h|(x)}{\omega(xy)}\right)^{p'}\ d\mu(x)
\leq \left(\frac{|h|(y)}{\omega(y)}\right)^{p'}\ \int_G\ \left(|h|(x)\ \omega\left(x^{-1}\right)\right)^{p'}\ d\mu(x) \\
\lesssim \left(\frac{|h|(y)}{\omega(y)}\right)^{p'}\ \int_G\ \left(\frac{|h|(x)}{\omega(x)}\right)^{p'}\ d\mu(x),
\end{gather*}

\noindent
where the last estimate follows from the diagonal boundedness of $\omega$. We conclude that 
$$\forall_{a.e.\ y\in G}\ 1\lesssim \left(\frac{|h|(y)}{\omega(y)}\right)^{p'}.$$

\noindent
Since $h\in L^{p'}_{\frac{1}{\omega}}(G)$, the above implies that $G$ is compact. 
\end{proof}

\section{$L^p_{\omega}(G)-$conjecture for nilpotent, locally compact groups}
\label{section:Nilpotentgroups}

At first, let us recall the \textit{principal structure theorem} for locally compact abelian groups, which can be found in \cite{Morris}, p. 86:

\begin{thm}
Every locally compact abelian group $G$ has an open subgroup, which is topologically isomorphic to $\reals^N \times K$, for some compact group $K$ and a non-negative integer $N$. 
\label{principalstructure}
\end{thm} 

We need two more results, which come from \cite{ANR}. They say, under what circumstances, the closedness under convolution is inherited by quotients and subgroups. 

\begin{thm}(Proposition 3.1 in \cite{ANR})\\
Let $H$ be a compact and normal subgroup of $G$. If $L^p_{\omega}(G)$ is closed under convolution, the so is $L^p_{\dot{\omega}}(G/N)$, where 
$$\dot{\omega}(xH) = \inf_{y\in H}\ \omega(xy).$$
\label{Lpquotient}
\end{thm}

\begin{thm}(Proposition 3.2 in \cite{ANR})\\
If $L^p_{\omega}(G)$ is closed under convolution and there is an infinite and discrete subgroup $H<G$ and compact, symmetric $K$ such that 
$$\forall_{\substack{h_1,h_2\in H\\ h_1\neq h_2}}\ K^2h_1K^2 \cap K^2h_2K^2 = \emptyset,$$

\noindent
then $l^p(H)$ is closed under convolution.
\label{Lpsubgroup}
\end{thm}

With such tremendous tools at hand, we are able to prove a generalization of Lemma 1.7 in \cite{RajagopalanII}. Our theorem reduces to Rajagopalan's result, when $\omega \equiv 1$. 

\begin{thm}
If $L^p_{\omega}(G)$ is closed under convolution, where $\omega$ is diagonally bounded, then the center $Z(G)$ is compact. 
\label{centeriscompact}
\end{thm}
\begin{proof}
Since $Z(G)$ is a closed subgroup of $G$ (comp. \cite{Pontrjagin}, p. 77), it is a locally compact abelian group. By Theorem \ref{principalstructure}, $Z(G)$ contains either an infinite, discrete, cyclic subgroup $H$ (if $N\geq 1$) or a compact, open subgroup $K$ (if $N=0$). For the sake of contradiction, we suppose that the former is true. 

Let $U_e\in\tau_G$ be such that $U_e \cap H = \{e\}$, this is possible due to discreteness. By continuity of multiplication, let $V_e \in\tau_G$ be a symmetric set such that \mbox{$V_e^2 \subset U_e$}. Suppose there exists $h\in H$ such that $V_e \cap hV_e \neq\emptyset$. This implies that $h\in V_e^2\subset U_e$. By the choice of $U_e$, we have $h = e$. In other words, we have established that 
$$\forall_{\substack{h_1,h_2\in H\\ h_1\neq h_2}}\ h_1V_e \cap h_2V_e = \emptyset.$$

\noindent
Using the continuity of multiplication and the fact that Hausdorff groups are automatically $T_{3\frac{1}{2}}$, we infer the existence of $W_e\in\tau_G$ such that $\overline{W_e}^4 \subset V_e$. 

At this stage, we are ready to apply Theorem \ref{Lpsubgroup}, concluding that $l^p(H)$ is closed under convolution. Next, Theorem \ref{LpconjectureforLCA} implies that $H$ must be compact. This is impossible due to $H$ being both infinite and discrete. The contradiction that we have reached means that there exists a compact, open subgroup $K$ of $Z(G)$. 

Let us consider a quotient group $Z(G)/K$, which is discrete by openess of $K$ (comp. \cite{HavinNikolski}, p. 172). Again, for the sake of contradiction, suppose that $Z(G)/K$ is infinite. By Theorem \ref{Lpquotient}, $L^p_{\dot{\omega}}(G/K)$ is closed under convolution. However, $Z(G/K)\supset Z(G)/K$, which means that $Z(G/K)$ contains an infinite, discrete subgroup. We have already shown above that the last two statements are incompatible. It follows that $Z(G)/K$ must be finite. 

At this point, we are able to conclude the proof. The subgroup $K$ is compact and we have just shown that $Z(G)/K$ is finite, so it is compact as well. Invoking Exercise 13 in \cite{Munkres}, p. 172 we are done. 
\end{proof}

Let $G$ be a nilpotent, locally compact group. This means that 
\begin{gather}
\{e\} = Z_0 \triangleleft Z_1 \triangleleft  \ldots \triangleleft Z_N = G,
\label{nilpotentchain}
\end{gather}

\noindent
where $Z_n$ denotes the $n-$th center of $G$, i.e.
$$Z_n = \big\{x \in G\ : \ \forall_{y\in G}\ [x,y]\in Z_{n-1}\big\} \hspace{0.4cm}\text{and}\hspace{0.4cm} Z_0 = \{e\}.$$

\noindent
The chain (\ref{nilpotentchain}) is sometimes referred to as the \textit{upper central series}. Obviously, $Z_1 = Z(G)$ and it is well-known that every $Z_n$ is a normal subgroup of $G$ - this is not completely trivial, as normality is not transitive in general. Moreover, it is an easy exercise to check that 
\begin{gather}
\forall_{n\in\naturals}\ Z_{n+1}/Z_n = Z(G/Z_n).
\label{Zn1Zn}
\end{gather}

We have finally reached the climax of the paper, the $L^p_{\omega}-$conjecture for nilpotent, locally compact groups. Note that the following result generalizes Theorem 1.10 in \cite{RajagopalanII}:

\begin{thm}
Let $G$ be a nilpotent, locally compact group and let $L^p_{\omega}(G)$ be closed under convolution, where $\omega$ is diagonally bounded. Then $G$ is compact. 
\label{nilpotentgroups}
\end{thm}
\begin{proof}
By Theorem \ref{centeriscompact}, we know that $Z_1=Z(G)$ is compact. Due to \mbox{Theorem \ref{Lpquotient}}, $L^p_{\dot{\omega}}(G/Z_1)$ is closed under convolution. By (\ref{Zn1Zn}) we have $Z_2/Z_1 = Z(G/Z_1)$, so again using Theorem \ref{centeriscompact}, we conclude that $Z_2/Z_1$ is compact. Invoking \mbox{Exercise 13} in \cite{Munkres}, p. 172 we establish that $Z_2$ is compact. Proceeding in a similar vein, we prove the compactness of $G$ in a finite number of steps. 
\end{proof}

As a final remark, we note that the above proof works for \textit{hypernilpotent}, locally compact groups (comp. \cite{DummitFoote}, p. 190-191). These are groups whose upper central series is countable and 
$$G = \bigcup_{n\in\naturals}\ Z_n.$$

\noindent
An analogous reasoning as in Theorem \ref{nilpotentgroups} shows that if $L^p_{\omega}(G)$ is closed under convolution ($\omega$ is still diagonally bounded), then $G$ must be $\sigma-$compact.


\begin{thebibliography}{9}
	\bibitem{ANR}
		Abtahi F., Nasr-Isfahani R., Rejali A. : \textit{Weighted $L^p-$conjecture for locally compact groups}, Periodica Mathematica Hungarica, Vol. 60, p. 1-11 (2010)
	\bibitem{ANRp2}
		Abtahi F., Nasr-Isfahani R., Rejali A. : \textit{Convolution on weighted $L^p-$spaces of locally compact groups}, Proceedings of the Romanian Academy, Series A, Vol. 13, p. 97-102 (2012)	
	\bibitem{DummitFoote}
		Dummit D. S., Foote R. M. : \textit{Abstract Algebra}, John Wiley and Sons, New Jersey, 2004
	\bibitem{GaudetGamlen} 
		Gaudet R. J., Gamlen J. L. \textit{An elementary proof of part of a classical conjecture},
Bulletin of Australian Mathematical Society, Vol. 3, p. 285-292 (1970)
	\bibitem{GelfandFomin}
		Gelfadn I. M., Fomin S. V. : \textit{Calculus of Variations}, Prentice Hall, New Jersey, 1963
	\bibitem{HavinNikolski}
		Havin V. P., Nikolski N. K. : \textit{Commutative Harmonic Analysis II}, Springer, Berlin, 1998
	\bibitem{Johnson}
		Johnson D. L. : \textit{A new proof of the $L^p-$conjecture for locally compact groups}, Colloquium Mathematicum, Vol. 48, p. 101-102 (1982)
	\bibitem{Kaniuth}
		Kaniuth E. : \textit{A Course in Commutative Banach Algebras}, Springer, New York, 2009
	\bibitem{Krukowski}
		Krukowski M. : \textit{Weights of reasonable growth and their application}, arXiv: 1709.00380 (2017)
	\bibitem{Kuznetsova}
		Kuznetsova Yu. N. : \textit{Invariant weighted algebras $L^{\omega}_p(G)$}, Mathematical Notes, Vol. 84, Nr. 4 (2008)
	\bibitem{Leptin}
		Leptin H. : \textit{On a certain invariant of a locally compact group}, Bulletin of American Mathematical Society, Vol. 72, Nr. 5, p. 870-874 (1966)
	\bibitem{Maligranda}
		Maligranda L. : \textit{Indices and interpolation}, Dissertationes Mathematicae, Vol. 234 (1985)	
	\bibitem{Morris}
		Morris S. A. : \textit{Pontryagin Duality and the Structure of Locally Compact Abelian Groups}, Cambridge University Press, Cambridge, 1977
	\bibitem{Munkres}
		Munkres J. R. : \textit{Topology}, Prentice Hall, New Jersey, 2000
	\bibitem{Pontrjagin}
		Pontrjagin I. : \textit{Topological groups}, Princeton University Press, Princeton, 1946
	\bibitem{Rajagopalandiscrete}
		Rajagopalan M. : \textit{On the $l^p-$spaces of a locally compact group}, Colloquium Mathematicum, Vol. 10, p. 49-52 (1963)
	\bibitem{RajagopalanI}
		Rajagopalan M. : \textit{$L_p-$conjecture for locally compact groups I}, Transactions of the American Mathematical Society, Vol. 125, p. 216-222 (1966) 
	\bibitem{RajagopalanII}
		Rajagopalan M. : \textit{$L_p-$conjecture for locally compact groups II}, Mathematische Annalen, Vol. 169, p. 331-339 (1967)
	\bibitem{RajagopalanZelazko}
		Rajagopalan M., \.Zelazko W. : \textit{$L_p-$conjecture for solvable locally compact groups}, Journal of the Indian Mathematical Society, Vol. 29 (1965)
	\bibitem{RickertL2}
		Rickert N. W. : \textit{Convolution of $L^2$ functions}, Colloquium Mathematicum, Vol. 19, p. 301-303 (1968)
	\bibitem{RickertLp}
		Rickert N. W. : \textit{Convolution of $L^p$ functions}, Proceedings of the American Mathematical Society, Vol. 18, p. 762–763 (1967)
	\bibitem{Saeki}
		Saeki S. : \textit{The $L^p$-conjecture and Young's inequality}, Illinois Journal of Mathematics, Vol. 34, Nr. 3 (1990)
	\bibitem{StegmanReiter}
		Stegman J. D., Reiter H. : \textit{Classical Harmonic Analysis and Locally Compact Groups}, Clarendon Press, Oxford, 2000
	\bibitem{Urbanik}
		Urbanik K. : \textit{A proof of a theorem of \.Zelazko on $L^p-$algebras}, Colloquium Mathematicum, Vol. 8, p. 121-123 (1961)
	\bibitem{ZelazkoonLpalgebras}
		\.Zelazko W. : \textit{On the algebras $L^p$ of locally compact groups}, Colloquium Mathematicum, Vol. 8, p. 115-120 (1961)	
\end{thebibliography}
\end{document}